\documentclass[11pt,a4paper]{article}
\usepackage{amsmath,amssymb,amsthm,amsfonts,latexsym,graphicx,subfigure}
\usepackage[all,import]{xy}
\usepackage[numbers,sort&compress]{natbib}

\usepackage{indentfirst}
\usepackage{fancyhdr,enumerate}
\usepackage{bbm,color}
\usepackage{tikz}
\usetikzlibrary{decorations.pathreplacing,calc}
\usetikzlibrary{shapes.geometric, arrows}
\usepackage{accents,cases}
\usepackage[T1]{fontenc} 
\usepackage{lmodern} 
\usepackage{multirow}
\usepackage[lined, ruled, linesnumbered, longend]{algorithm2e}
\usepackage{tikz}
\usetikzlibrary{shapes.geometric, arrows}
\usepackage[labelsep=space]{caption}

\usepackage{array,float}
\newcommand{\PreserveBackslash}[1]{\let\temp=\\#1\let\\=\temp}
\newcolumntype{C}[1]{>{\PreserveBackslash\centering}p{#1}}
\newcolumntype{R}[1]{>{\PreserveBackslash\raggedleft}p{#1}}
\newcolumntype{L}[1]{>{\PreserveBackslash\raggedright}p{#1}}

\makeatletter
\def\wbar{\accentset{{\cc@style\underline{\mskip8mu}}}}

\makeatother

\def\d{\mathrm{d}}

\theoremstyle{plain}
\newtheorem{theorem}{Theorem}

\newtheorem{defn}{Definition}[section]
\newtheorem{Conj}{Conjecture}[section]
\newtheorem{lemma}{Lemma}

\newtheorem{pro}{Proposition}[section]

\allowdisplaybreaks

\begin{document}

\title{Equiangular lines via nodal domains}
\author{Chuanyuan Ge\footnotemark[1]\and Shiping Liu\footnotemark[2] }

\footnotetext[1]{School of Mathematics and Statistics, Fuzhou University, Fuzhou 350108, China.\\
Email addresses:
{\tt gechuanyuan@mail.ustc.edu.cn} }
\footnotetext[2]{School of Mathematical Sciences, 
University of Science and Technology of China, Hefei 230026, China. \\
Email addresses:
{\tt spliu@ustc.edu.cn}
}
\date{}\maketitle
\begin{abstract}
For given $\Delta>0$ and $0<\lambda<3/\sqrt{2}$, we show that the maximum multiplicity that $\lambda$ can appear as the second largest eigenvalue of a connected graph with maximum degree at most $\Delta$ is $O_{\Delta,\lambda}(1)$. This result answers a question due to Jiang, Tidor, Yao, Zhang and Zhao [Question 6.4, Ann. of Math. (2) 194 (2021), no. 3, 729-743] in the case of $0<\lambda<3/\sqrt{2}$, and consequently leads to improvements in their results on equiangular lines. Our proof is based on the concept of nodal domains of eigenfunctions. Indeed, we establish a multiplicity estimate in terms of maximum degree and cyclomatic number of the graph, via a novel construction of eigenfunctions with large number of nodal domains.

\end{abstract}

\section{Introduction}
A collection of lines that traverse the origin within $\mathbb{R}^d$ is called equiangular lines if the angle between any two of them is constant. Equiangular lines, along with their various variants, are closely related to various research topics, see \cite{Balla-Draxler-Keevash-Sudakov-18,Jiang-Tidor-Yao-Zhang-Zhao-21,balla2024equiangular-exponential-regime} and the references therein.

An interesting problem proposed by van Lint and Seidel \cite{van-Seidel-elliptic-geometry} is to determine the maximum number $N(d)$ of equiangular lines in $\mathbb{R}^d$. This has become one of the founding problems of algebraic graph theory. An upper bound $ N(d)\leq \binom{d+1}{2}$ was shown by Gerzon, see \cite{Lemmens-Seidel-73}. On the other hand, de Caen \cite{de-00} showed that $ N(d) \geq \frac{2}{9}(d+1)^2$ for $d=6\cdot4^i-1$ by constructing explicit collections of equiangular lines. 
We refer to \cite{Barg-Yu-14,Greaves-Koolen-Munemasa-16,Jedwab-Wiebe-15} for other constructions.

For $\alpha\in (0,1)$, let $N_{\alpha}(d)$ be the maximum number of equiangular lines in $\mathbb{R}^d$ with fixed angle $\arccos\alpha$. In 1973, Lemmens and Seidel \cite{Lemmens-Seidel-73} posed the problem of determining $N_{\alpha}(d)$. In fact, they proved that $N_{\frac{1}{3}}(d)=2d-2$ for $d\geq 15$. Furthermore, they conjectured  \cite[Conjecture 5.8]{Lemmens-Seidel-73} that $N_{1/5}(d)=\left\lfloor  \frac{3(d-1)}{2} \right\rfloor$ for $d\geq 185$, and $N_{1/5}(d)=276$ for $23\leq d\leq 185$. Neumaier \cite{Neumaier-89} proved in 1989 that $N_{1/5}(d)$ equals $\left\lfloor \frac{3(d-1)}{2} \right\rfloor$ for sufficiently large $d$. In 2022, Cao, Koolen, Lin, and Yu \cite{Cao-Koolen-Lin-Yu} confirmed this conjecture completely.

For general $\alpha$, Neumaier showed that $N_{\alpha}(d)\leq 2d$ unless $\frac{1}{\alpha}$ is an odd integer, see \cite{Lemmens-Seidel-73}. 
In 2016, Bukh \cite{Bukh-16} proved that $N_{\alpha}(d)\leq c_{\alpha}d$, where $c_{\alpha}=2^{O(1/\alpha^2)}$. Later, Balla, Dr\"axler, Keevash, and Sudakov \cite{Balla-Draxler-Keevash-Sudakov-18} made a breakthrough on this problem by proving $\limsup_{d\to \infty }N_{\alpha}(d)/d\leq 1.93$ for any $\alpha\neq \frac{1}{3}$. Following the result in \cite{Lemmens-Seidel-73} and \cite{Neumaier-89}, Bukh \cite{Bukh-16} conjectured that $N_{1/(2k-1)}(d) = kd/(k-1)+O_k(1)$ as $d$ tends to $\infty$. To solve this conjecture, Jiang and Polyanskii \cite{Jiang-Polyanskii-20} introduced the following spectral graph theoretic quantity. 

\begin{defn}[Spectral radius order]
    Given any $\lambda\in \mathbb{R}$, the spectral radius order $\kappa(\lambda)$ is defined as follows:$$\kappa(\lambda):=\min\{n\in Z:\text{ there exists a graph $G=(V,E)$ such that $\lambda_1(G)=\lambda$ and $|V|=n$ }\},$$
    where $\lambda_1(G)$ is the largest eigenvalue of the adjacency matrix of $G$. If there does not exist any graph $G$ such that $\lambda_1(G)=\lambda$, then we set $\kappa(\lambda)=\infty$.
\end{defn}
Jiang and Polyanskii \cite{Jiang-Polyanskii-20}
proposed the following conjecture
\[\lim_{d\to \infty}\frac{N_{\alpha}(d)}{d}=\frac{\kappa(\lambda)}{\kappa(\lambda)-1},\,\,\text{ with}\,\,\lambda=\frac{1-\alpha}{2\alpha},\] and confirmed it for the case $\lambda<\sqrt{2+\sqrt{5}}$. 
Since $\kappa(m)=m+1$ for any $m\in Z^+$, Jiang and Polyanskii's conjecture is a stronger version of Bukh's conjecture.
In 2021, Jiang, Tidor, Yao, Zhang, and Zhao completely solved Jiang and Polyanskii's conjecture by proving the following theorem \cite[Theorem 1.2]{Jiang-Tidor-Yao-Zhang-Zhao-21}.
\begin{theorem}\label{thm:k-lambad}
    Let $\alpha\in (0,1)$. Define $\lambda=\frac{1-\alpha}{2\alpha}$ and
$\kappa=\kappa(\lambda)$ as its spectral radius order. 
\begin{itemize}
    \item [(a)] If $\kappa<\infty$, then $N_{\alpha}(d)=\left\lfloor \frac{(d-1)\kappa}{\kappa-1}\right\rfloor$ for $d>2^{2^{C\kappa/\alpha}}$, where $C$ is an absolute constant. 
    \item  [(b)] If $\kappa=\infty$, then $N_{\alpha}(d)=d+o_{\alpha}(d)$ as $d\to \infty$.
\end{itemize}
\end{theorem}
There are two directions to improve the above theorem. One is to lower down the bound for $d$ in $(a)$, the other is to improve the bound for $N_\alpha(d)$ in $(b)$. In the first direction,   Balla \cite{balla2021equiangular} improved the lower bound condition to $d>2^{2^{C_0\kappa\log(\frac{1}{\alpha})}}$, where $C_0$ is an absolute constant. Balla and Buci{\'c} \cite{balla2024equiangular-exponential-regime} showed that $N_{\alpha}(d)\leq\left\lfloor(d-1)(1+\frac{2\alpha}{1-\alpha})\right\rfloor$ for $d\geq 2^{\frac{1}{\alpha^{20}}}$. 
 Since $\kappa(\lambda)\geq \lambda+1$ for any $\lambda\in (0,+\infty)$, we have $\left\lfloor\frac{(d-1)\kappa}{\kappa-1}\right\rfloor\leq \left\lfloor(d-1)(1+\frac{2\alpha}{1-\alpha})\right\rfloor$. In the second direction, 
 Jiang et al. proposed the following conjecture \cite[Conjecture 6.1]{Jiang-Tidor-Yao-Zhang-Zhao-21} and confirmed it for any $\lambda$ that is not a totally real algebraic integer.
 \begin{Conj}\label{conj:infty}
     Let $\alpha$, $\lambda$, and $\kappa$ be defined as in Theorem \ref{thm:k-lambad}. If $\kappa=\infty$, then $N_{\alpha}(d)=d+O_{\alpha}(1)$.
 \end{Conj}
However, Schildkraut \cite{schildkraut-2023equiangular} constructed counterexamples to show that the above conjecture fails for infinitely many $\lambda$. It remains open to figure out the set of $\lambda$ for which Conjecture \ref{conj:infty} holds.

In this article, we improve Theorem \ref{thm:k-lambad} in the two directions described above for the case $\lambda<\frac{3}{\sqrt{2}}$. 

 \begin{theorem}\label{thm:main}
       Let $\alpha\in (0,1)$. Define $\lambda=\frac{1-\alpha}{2\alpha}$ and
$\kappa=\kappa(\lambda)$ as its spectral radius order.
       \begin{itemize}
           \item [(a)] If $\kappa<\infty$ and $\lambda\leq \lambda_0<\frac{3}{\sqrt{2}}$, then $N_{\alpha}(d)=\left\lfloor \frac{(d-1)\kappa}{\kappa-1}\right\rfloor$ for $d>C_{\lambda_0}\kappa$, where $C_{\lambda_0}$ is a constant depending only on $\lambda_0$. 
           \item [(b)] If $\kappa=\infty$ and $\lambda\leq \lambda_0<\frac{3}{\sqrt{2}}$, then $N_{\alpha}(d)=d+O_{\lambda_0}(1)$.
       \end{itemize}      
 \end{theorem}
In particular, Theorem \ref{thm:main} $(b)$ tells that Conjecture \ref{conj:infty} holds for any $\alpha>\frac{1}{3\sqrt{2}+1}$.

We prove Theorem \ref{thm:main} by establishing the following eigenvalue multiplicity estimation. 
For given $\Delta, \lambda>0$, Jiang et al. asked \cite[Question 6.4]{Jiang-Tidor-Yao-Zhang-Zhao-21} what is the maximum multiplicity that $\lambda$ can appear as the second largest eigenvalue of a connected $n$-vertex graph with maximum degree at most $\Delta$. Moreover, they noticed that Conjecture \ref{conj:infty} holds for any $\lambda$ with maximum multiplicity $O_{\Delta,\lambda}(1)$. We show that any $\lambda<\frac{3}{\sqrt{2}}$ has maximum multiplicity $O_{\Delta,\lambda}(1)$.

\begin{theorem}\label{thm:multi}
  Let $G$ be a connected graph,  $\lambda_2(G)$ be its second largest adjacency eigenvalue, and $m_{\lambda_2(G)}$ be the multiplicity of $\lambda_2(G)$. If $\lambda_2(G)\leq \lambda<\frac{3}{\sqrt{2}}$ and the maximum degree of $G$ does not exceed $\Delta$,  then we have $$m_{\lambda_2(G)}\leq 2n_{\lambda}\Delta^{n_{\lambda}+3}\left(1+\Delta+\Delta^2\right),$$
where $n_\lambda$ is a constant depending only on $\lambda$.
\end{theorem}
 For a precise definition of $n_\lambda$, we refer to Definition \ref{def:n}. It is interesting to ask for the maximum number $a>0$ such that for any $\lambda\in (0,a)$, $\lambda$ can appear as the second largest eigenvalue of a connected graph with maximum degree at most $\Delta$ is $O_{\Delta, \lambda}(1)$.

The proof of Theorem \ref{thm:multi} is based on the concept of nodal domains. The idea of deriving upper bounds of eigenvalue multiplicity using nodal domains dates back to Cheng's pioneering work \cite{Cheng-76} in the context of Riemann surfaces. Recall that Cheng's multiplicity estimate is given in terms of the genus of the Riemann surface. Van Der Holst \cite{van-der-Holst-95} extended Cheng's approach to graphs, proving that the multiplicity of the second smallest Laplacian eigenvalue of a maximal planar graph is at most $3$. Following Cheng's idea, Lin, Lippner, Mangoubi, and Yau \cite{lin2010nodalmulti}  derived
an upper bound for the multiplicity of the $k$-th Laplacian eigenvalue in terms of the graph genus when the graph is $3$-connected and satisfies a quadratic volume-growth condition.

Unlike the previous extensions of Cheng's idea in terms of genus, we show a multiplicity estimate in terms of the cyclomatic number of a graph (see Lemma \ref{lemma:multi tree}), via a novel construction of eigenfunctions with a large number of nodal domains. Furthermore, we show that, for a graph $G$ with $\lambda_2(G)<\frac{3}{\sqrt{2}}$, the cyclomatic number of $G$ is small after removing a vertex subset, whose size is bounded by a constant depending only on $\lambda_2(G)$. Then we finish the proof of Theorem \ref{thm:multi} by Cauchy's Interlace Theorem. 

We note that the methods developed in the graph setting also offer valuable insights for research in continuous spaces. For example, Letrouit and Machado \cite{Letroui-Machado-24} made progress on a conjecture of Colin de Verdière \cite{Colin-87} about the multiplicity of eigenvalues by extending the ideas of Jiang et al. \cite{Jiang-Tidor-Yao-Zhang-Zhao-21} to the setting of negatively curved surfaces. It is an interesting question to ask whether Lemma \ref{lemma:multi tree} admits a continuous analogue or not. 

 The paper is structured as follows. In Section \ref{sec:pre}, we collect preliminaries on the adjacency matrix and define $5$ classes of graphs that will be used in the proof of our main results. In Section \ref{sec:radius-graph}, we analyze the property of the spectral radius of graphs defined in Section \ref{sec:pre}. We prove Theorems \ref{thm:multi} and \ref{thm:main} in Section \ref{sec:multi} and Section \ref{sec:equiangular}, respectively.

\section{Notation and preliminaries}
\label{sec:pre}
In this paper, we only consider finite simple undirected graphs.

Let $G=(V,E)$ be a graph. 
Given a subset $V_0\subset V$, we denote by $G[V_0]$ the subgraph of $G$ induced by $V_0$, that is, $G[V_0]:=(V_0,E_0)$ with $E_0:=\{\{v,w\}\in E:v,w\in V_0\}$. For any vertex $v\in V$, we use $d_G(v)$ to denote the degree of $v$ in $G$. We define
\[\Delta_G:=\max_{v\in V}d_G(v)\] 
to be the maximum degree of $G$. The cyclomatic number of $G$ is defined as 
\[\ell_G:=|E|-|V|+1.\]

We denote by $\mathrm{diam}(G)$ the diameter of $G$, and by $\mathfrak{g}(G)$ the girth of $G$. For any $V_0\subset V$, we define $B^{G}_k(V_0)$ as follows
$$ B^{G}_k(V_0):=\{v\in V:\mathrm{dis}_G(x,V_0)\leq k\},$$
where $\mathrm{dis}_G(v,V_0)$ is the combinatorial distance between $v$ and $V_0$ in $G$. 

We use $A_G$ to denote the adjacency matrix of $G$ and use
\[\lambda_1(G)\geq \lambda_2(G)\geq \cdots\geq \lambda_{|V|}(G)\]
to denote all eigenvalues of $A_G$ with multiplicity. For any matrix $A$ and its eigenvalue $\lambda$, we use $m_{\lambda}^A$ to denote the multiplicity of $\lambda$ in $A$. If no confusion arises, we use $m_{\lambda}$ to denote $m_{\lambda}^A$ for ease of notation.

We introduce the following $5$ classes of graphs.
\begin{defn}\label{def:4graph}
    Let $p,q,\ell\in \mathbb{Z}^+$ be positive integers. We define the graphs $\mathcal{P}(p,q,\ell)$, $\mathcal{D}(p,q)$, $\mathcal{B}(p,q,\ell)$, $\mathcal{T}(p,q,\ell)$ and $\mathcal{H}(p,q)$ as follows.
\begin{itemize}
    \item[$(a)$] For $p,q\geq 2$ and $\ell\geq 1$, we define $\mathcal{P}(p,q,\ell)$ as the graph consisting of $3$ paths of length $p$, $q$ and $m$, respectively, where the three paths share exactly their endpoints.
        \item [$(b)$] For $p,q\geq 3$, we define $\mathcal{D}(p,q)$ as the graph consisting of two cycles of lengths $p$ and $q$, respectively, where the two cycles share exactly one vertex. 
        \item [$(c)$]  For $p,q\geq 3$ and $\ell\geq 1$, we define $\mathcal{B}(p,q,\ell)$ as the graph consisting of a path of length $\ell$ connecting two cycles of lengths $p$ and $q$, respectively.
        \item [$(d)$] For $p,q,\ell\geq 1$, we define $\mathcal{T}(p,q,\ell)$ as the graph consisting of three paths of lengths $\ell$, $p$ and $q$, respectively, where the three paths share exactly their initial vertex. 
            \item [$(e)$] For $p\geq 3$ and $q\geq 1$, we define $\mathcal{H}(p,q)$ as the graph consisting of a cycle of length $p$ and a path of length $q$, where the cycle and the path share exactly one vertex. 
    \end{itemize}
\end{defn}
The graphs $\mathcal{P}(2,2,3),\,\mathcal{D}(3,3),\,\mathcal{B}(3,3,2),\,\mathcal{T}(3,3,3)$ and $\mathcal{H}(3,2)$ are depicted in Figure \ref{fig}.

\begin{figure}[!htp]
	\centering
	\tikzset{vertex/.style={circle, draw, fill=black!20, inner sep=0pt, minimum width=3pt}}	

	\begin{tikzpicture}[scale=1.0]
 \draw    (-10,0) -- (-9,0) node[midway, above, black]{$ $} -- (-8,0) node[midway, above, black]{$ $} --(-7,0) node[midway, above, black]{$ $};

   \draw    (-10,0) -- (-8.5,-1) node[midway, below, black]{$ $}-- (-7,0) node[midway, above, black]{$ $} ;

      \draw    (-10,0) -- (-8.5,1) node[midway, above, black]{$ $}-- (-7,0) node[midway, above, black]{$ $} ;
        
		\node at (-10,0) [vertex, label={[label distance=0mm]90: \small $ $}, fill=black] {};
           
		\node at (-9,0) [vertex, label={[label distance=0mm]90: \small $ $}, fill=black] {};
           
		\node at (-8,0) [vertex, label={[label distance=0mm]90: \small $ $}, fill=black] {};
           
		\node at (-7,0) [vertex, label={[label distance=0mm]90: \small $ $}, fill=black] {};

		\node at (-8.5,1) [vertex, label={[label distance=0mm]90: \small $ $}, fill=black] {};
           
		\node at (-8.5,-1) [vertex, label={[label distance=5mm]270: \small (a) $\mathcal{P}(2,2,3) $}, fill=black] {};

 \draw    (-6,1) -- (-5,0) node[midway, above, black]{$ $} -- (-4,-1) node[midway, above, black]{$ $}-- (-4,1) node[midway, above, black]{$ $}-- (-6,-1) node[midway, above, black]{$ $} --(-6,1) node[midway, above, black]{$ $};

\node at (-6,1) [vertex, label={[label distance=0mm]90: \small $ $}, fill=black] {};

\node at (-5,0) [vertex, label={[label distance=15mm]270: \small (b) $\mathcal{D}(3,3) $}, fill=black] {};

\node at (-4,1) [vertex, label={[label distance=0mm]90: \small $ $}, fill=black] {};

\node at (-4,-1) [vertex, label={[label distance=0mm]90: \small $ $}, fill=black] {};

\node at (-6,-1) [vertex, label={[label distance=0mm]90: \small $ $}, fill=black] {};

 \draw    (-3,1) -- (-2,0) node[midway, above, black]{$ $} -- (-3,-1) node[midway, above, black]{$ $}-- (-3,1) node[midway, above, black]{$ $};

  \draw    (-2,0) -- (0,0) node[midway, above, black]{$ $} -- (1,1) node[midway, above, black]{$ $}-- (1,-1) node[midway, above, black]{$ $}-- (0,0) node[midway, above, black]{$ $};
  
\node at (-2,0) [vertex, label={[label distance=0mm]90: \small $ $}, fill=black] {};

\node at (0,0) [vertex, label={[label distance=0mm]90: \small $ $}, fill=black] {};

\node at (1,1) [vertex, label={[label distance=0mm]90: \small $ $}, fill=black] {};

\node at (1,-1) [vertex, label={[label distance=0mm]90: \small $ $}, fill=black] {};

\node at (0,0) [vertex, label={[label distance=0mm]90: \small $ $}, fill=black] {};

\node at (-3,-1) [vertex, label={[label distance=0mm]90: \small $ $}, fill=black] {};

\node at (-3,1) [vertex, label={[label distance=0mm]90: \small $ $}, fill=black] {};

\node at (-1,0) [vertex, label={[label distance=15mm]270: \small (c) $ \mathcal{B}(3,3,2)$}, fill=black] {};
  \draw    (-9,-4) -- (-8,-4) node[midway, above, black]{$ $} -- (-7,-4) node[midway, above, black]{$ $}-- (-6,-4) node[midway, above, black]{$ $};

  \draw    (-9,-4) -- (-8,-3) node[midway, above, black]{$ $} -- (-7,-3) node[midway, above, black]{$ $}-- (-6,-3) node[midway, above, black]{$ $};

  \draw    (-9,-4) -- (-8,-5) node[midway, above, black]{$ $} -- (-7,-5) node[midway, above, black]{$ $}-- (-6,-5) node[midway, above, black]{$ $};

\node at (-9,-4) [vertex, label={[label distance=15mm]270: \small $ $}, fill=black] {};

\node at (-8,-3) [vertex, label={[label distance=15mm]270: \small $ $}, fill=black] {};

\node at (-7,-3) [vertex, label={[label distance=15mm]270: \small $ $}, fill=black] {};

\node at (-6,-3) [vertex, label={[label distance=15mm]270: \small $ $}, fill=black] {};

  \node at (-8,-4) [vertex, label={[label distance=15mm]270: \small $ $}, fill=black] {};

\node at (-7,-4) [vertex, label={[label distance=15mm]270: \small $ $}, fill=black] {};

  \node at (-6,-4) [vertex, label={[label distance=15mm]270: \small $ $}, fill=black] {};

  \node at (-6,-5) [vertex, label={[label distance=15mm]270: \small $ $}, fill=black] {}; 

   \node at (-7,-5) [vertex, label={[label distance=5mm]270: (d) \small $\mathcal{T}(3,3,3) $}, fill=black] {};

    \node at (-8,-5) [vertex, label={[label distance=15mm]270: \small $ $}, fill=black] {};

  \draw    (-3,-3) -- (-2,-4) node[midway, above, black]{$ $} -- (-3,-5) node[midway, above, black]{$ $}-- (-3,-3) node[midway, above, black]{$ $};

   \draw    (-2,-4) -- (0,-4) node[midway, above, black]{$ $};

\node at (-3,-3) [vertex, label={[label distance=15mm]270: \small $ $}, fill=black] {};

\node at (-3,-5) [vertex, label={[label distance=15mm]270: \small $ $}, fill=black] {};

\node at (-2,-4) [vertex, label={[label distance=15mm]270: \small $ $}, fill=black] {};

\node at (0,-4) [vertex, label={[label distance=15mm]270: \small $ $}, fill=black] {};

\node at (-1,-4) [vertex, label={[label distance=15mm]270: \small (e) $\mathcal{H}(3,2)$}, fill=black] {};
	\end{tikzpicture}
	\caption{Illustrations of Definition \ref{def:4graph}}
	\label{fig}
\end{figure}
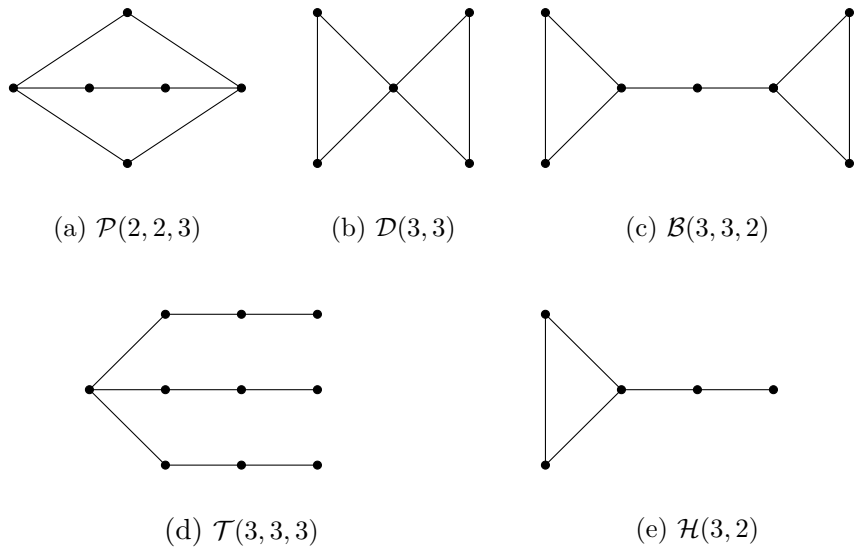

\begin{defn}
    A path $\{x_{i}\}_{i=1}^k\subset V$ of a graph $G$ is an end path if $d_G(x_i)=2$ for $1<i<k$ and $d_G(x_k)=1$.
\end{defn}
We recall two lemmas concerning the monotonicity of the spectral radius.
\begin{lemma}[{\cite[Lemma 3.1]{Hoffman-72}}]\label{lemma:mono1}
    Let $G'$ be a subgraph of $G$. We have $\lambda_1(G')\leq \lambda_1(G).$ If $G$ is connected and $G'$ is a proper subgraph of $G$, then we have $\lambda_1(G')<\lambda_1(G)$.
\end{lemma}
\begin{lemma}[{\cite[Proposition 4.1]{Hoffman-72}}]\label{lemma:mono2}
    Let $G=(V,E)$ be a connected graph with $\lambda_1(G)>2$ and $e=\{x,y\}$ be an edge which is not in any end path of $G$. We define $G'$ to be a graph obtained by removing $e$ from $G$ and adding a new vertex that is adjacent only to $x$ and $y$. Then we have $$\lambda_1(G')\leq \lambda_1(G).$$
\end{lemma}

We recall the definition of strong nodal domains \cite{BLS07}. In the proof of our main theorem, we strategically utilized this concept.
\begin{defn}[Stong nodal domains]
   Let $G=(V,E)$ be a graph. A strong nodal domain of a function $f:V\to \mathbb{R}$ is a maximal connected component of the subgraph $G_0=(V_0,E_0)$, where $V_0=\{x\in V:f(x)\neq 0\}$ and $E_0=\{\{x,y\}\in E:f(x)f(y)>0\}$. We use $\mathfrak{S}_G(f)$ to denote the number of strong nodal domains of $f$ on $G$.
\end{defn}
  To conclude this section, we recall Cauchy’s interlace theorem. A direct proof of this theorem can be found in \cite{fisk2005interlacing-proof}.
\begin{theorem}[Cauchy’s Interlace Theorem]\label{thm:interlacing}
    Let $A$ be an $n\times n$ symmetric matrix, and
B be a $m\times m$ principal submatrix of $A$ with $m<n$. Denote by $\lambda_1(A)\geq \lambda_2(A)\geq \cdots\geq \lambda_n(A)$ all eigenvalues of $A$ and by $\lambda_1(B)\geq \lambda_2(B)\geq \cdots\geq \lambda_m(B)$ all eigenvalues of $B$. For any $1\leq k\leq n-m $, We have $$\lambda_k(A)\geq \lambda_k(B)\geq \lambda_{k+m-n}(A).$$
\end{theorem}
\section{Spectral radius of special graphs}\label{sec:radius-graph}
In this section, we study the spectral radius of those graphs in Definition \ref{def:4graph}. Furthermore, we introduce the constant $n_{\lambda}$ for $\lambda\in(0,\frac{3}{\sqrt{2}})$, which is needed in our Theorem \ref{thm:multi}. 

\begin{pro}\label{pro:T}
   The eigenvalue $\lambda_1(\mathcal{T}(\ell,\ell,\ell))$ is strictly increasing with respect to $\ell$. Moreover, we have $$\lim_{\ell\to \infty}\lambda_1(\mathcal{T}(\ell,\ell,\ell))=\frac{3}{\sqrt{2}}.$$
\end{pro}

\begin{proof}
The first statement follows directly from Lemma \ref{lemma:mono1}. Since $\lambda_1(\mathcal{T}(\ell,\ell,\ell))\leq \Delta_{\mathcal{T}(\ell,\ell,\ell)}= 3$, the limit $\lim_{\ell\to\infty}\lambda_1\left(T(\ell,\ell,\ell)\right)$ exists. To compute this limit, we define a matrix 
\[B_1=
\begin{bmatrix}
0 & 3 \\
1 & 0
\end{bmatrix},
\]
and recursively define
\[
B_i=\begin{bmatrix}
\mathbf{B}_{i-1} &\begin{matrix}0\\\vdots\\0\\1\end{matrix}\\
\begin{matrix}0&\cdots&0&1\end{matrix}&\begin{matrix}0\end{matrix}
\end{bmatrix}
,\] 
for any $i\geq 2$.

Let $f\in \mathbb{R}^{\ell+1}$ be an eigenvector of $B_{\ell}$ corresponding to its spectral radius $\rho(B_{\ell})$. Due to the Perron-Frobenius theorem, $f$ is a positive vector, that is, every coordinate of $f$ is positive. We denote by $V_{\ell}$ the vertex set of $\mathcal{T}(\ell,\ell,\ell)$, and by $x_0$ the unique vertex of degree $3$ in $\mathcal{T}(\ell,\ell,\ell)$.
Define a function $g:V_{\ell}\to \mathbb{R}$ to be $g(x):=f\left(d(x,x_0)+1\right)$, where $f(i)$ is the $i$-th coordinate of $f$ and $d(x,x_0)$ is the combinatorial distance  between $x$ and $x_0$ in  $\mathcal{T}(\ell,\ell,\ell)$. We directly check that $g$ is an eigenfunction of $A_{\mathcal{T}(\ell,\ell,\ell)}$ corresponding to the eigenvalue $\rho(B_\ell)$. Since $g$ is positive on every vertex of $V_{\ell}$, we have $\rho(B_{\ell})=\lambda_1(\mathcal{T}(\ell,\ell,\ell)$  by the Perron-Frobenius theorem. In particular, we have $$\lim_{\ell\to\infty}\rho\left(B_{\ell}\right)=\lim_{\ell\to \infty}\lambda_1\left(\mathcal{T}(\ell,\ell,\ell)\right).$$

Let $P_{\ell}(t)$ denote the characteristic polynomial of $B_{\ell}$ for $\ell\geq 1$. Define $P_0(t)=t$. Expanding the determinant of $tI-B_{\ell}$ along the $(\ell+1)$-th row, we have $P_{\ell}(t)=tP_{\ell-1}(t)-P_{\ell-2}(t)$ for $\ell\ge 2$. According to this recursion, we have \begin{equation}\label{eq:1}
    P_{\ell}(t)=\left(\frac{\theta(t)P_1(t)-P_0(t)}{\theta^2(t)-1}\right)\theta^\ell(t)+\left(\frac{P_0(t)-\frac{P_1(t)}{\theta(t)}}{1-\theta^{-2}(t)}\right)\theta^{-\ell}(t),\end{equation}
where \[\theta(t)=\frac{t+\sqrt{t^2-4}}{2}.\] By Equation (\ref{eq:1}) and the fact that $\lambda_1(\mathcal{T}(\ell,\ell,\ell))>2$ for $\ell>2$, we have that $\lim_{\ell\to \infty}\lambda_1(\mathcal{T}(\ell,\ell,\ell))$ equals the largest root of $\theta(t)P_1(t)-P_0(t)$. Inserting $P_0(t)=t$ and $P_1(t)=t^2-3$, we have that the largest root of $\theta(t)P_1(t)-P_0(t)$ is $\frac{3}{\sqrt{2}}$. This completes the proof of this proposition.
\end{proof}

\begin{pro}\label{pro:two-cycle}
Let $G$ be a graph from the classes defined in Definition \ref{def:4graph} $(a)$, $(b)$, or $(c)$. Then we have $$\lambda_1(G)\geq \frac{3}{\sqrt{2}}.$$
\end{pro}
\begin{proof} 
We present only the proof for the case $G=\mathcal{P}(p,q,\ell)$ for $p,q\geq2$ and $\ell\geq 1$. The other cases $G=\mathcal{D}(p,q)$ with $p,q\geq 3$ and $G=\mathcal{B}(p,q,\ell)$ with $p,q\geq 3$ and $\ell\geq 1$ can be shown similarly.

 We prove this by contradiction. Assume that $\lambda_1(\mathcal{P}(p,q,\ell))<\frac{3}{\sqrt{2}}$. By Proposition \ref{pro:T}, there exists a large enough integer $k$  such that $\lambda_{1}(\mathcal{T}(k,k,k))> \lambda_1(\mathcal{P}(p,q,\ell))$. Then, we estimate 
 \begin{equation*}
     \begin{aligned}
         \lambda_1(\mathcal{P}(p,q,\ell))\geq \lambda_1(\mathcal{P}(pk+1,qk+1,\ell k+1))\geq \lambda_1(\mathcal{T}(pk,qk,\ell k))\geq \lambda_{1}(\mathcal{T}(k,k,k))>\lambda_1(\mathcal{P}(p,q,\ell)),
     \end{aligned}
 \end{equation*}
 where we applied Lemma \ref{lemma:mono2} in the first inequality, and Lemma \ref{lemma:mono1} in the second inequality above. This is a contradiction.
\end{proof}
Based on Proposition \ref{pro:T}, we introduce the following definition.
\begin{defn}\label{def:n}
    For any $\lambda<\frac{3}{\sqrt{2}}$, we define $n_{\lambda}$ as the smallest integer such that \[\lambda_1\left(\mathcal{T}\left(n_{\lambda},n_{\lambda},n_{\lambda}\right)\right)>\lambda.\] 
\end{defn}

\section{Proof of Theorem \ref{thm:multi}}
\label{sec:multi}
To establish Theorem \ref{thm:multi}, we first prove the following upper bound estimate for the multiplicity of adjacency eigenvalues in terms of the maximum degree and the cyclomatic number. 
\begin{lemma}\label{lemma:multi tree}
     Let $G=(V,E)$ be a connected graph, we have $m_{\lambda_k(G)}\leq (k-1)\Delta_G+\ell_G$.
\end{lemma}
To prove this lemma, we need the following upper bound estimate of strong nodal domains due to Lin, Lippner, Mangoubi and Yau \cite{lin2010nodalmulti}.
\begin{theorem}[{\cite[Theorem 1.1]{lin2010nodalmulti}}]\label{thm:nodal upper}
    Let $G$ be a connected graph and $f$ be an eigenfunction corresponding to $\lambda_k(G)$. Then we have $\mathfrak{S}_G(f)\leq (k-1)\Delta_G $.
\end{theorem}

  \begin{proof}[Proof of Lemma \ref{lemma:multi tree}]
  
We label the vertices by $v_1,v_2,\ldots,v_{n}$ and write $m:=m_{\lambda_k(G)}$ for simplicity. Without loss of generality, we assume that there exists a base $\{f_i\}_{i=1}^m$ of the eigenspace of $A_G$ corresponding to $\lambda_k(G)$ satisfying the following property: For each $1\leq i\leq m$, we have that $f_i(v_i)=1$ and $f_i(v_j)=0$ for any $j\in\{1,\ldots,m\}\setminus\{i\}$.
 
Let $T$ be a spanning tree of $G$. We denote by $d_T$ the combinatorial distance of $T$.
Without loss of generality, we further assume  that for any $1\leq i\leq m$, there does not exist a vertex $v$ such that  $f_i(v)\neq 0$ and $d_T(v,v_1)< d_T(v_i,v_1)$. Indeed, if there exists a vertex $v$ such that $f_i(v)\neq 0$ and $d_T(v,v_1)< d_T(v_i,v_1)$, we exchange the labels of $v$ and $v_i$ and re-normalized the function $f_i$.

Pick $v_1$ as the root of the $T$. Denote $r:=\mathrm{diam}(T)$. For any vertex $v\in T$, we use  $v^F$ to denote the father of $v$ in $T$.  We define a sequence of functions $\{g_k\}_{k=0}^r$ iteratively. Set $g_0=f_0$. For $1\leq k\leq r$, we define $$g_k=g_{k-1}+\sum_{\substack{1\leq j\leq m\\d_T(v_j,v_1)=k}}c_{k,j}f_j,$$
where $c_{k,j}=1$ when $g_{k-1}(v_j^F)\leq 0$, and $c_{k,j}=-1$ when $g_{k-1}(v_j^F)> 0$. It is worth noting that by our definition and assumption on $\{f_i\}_{i=1}^r$, we have $g_{k_0}(v)=g_k(v)$ for $1\leq k< k_0\leq r$ and any vertex $v\in V_0$ with $d_T(v,v_1)\leq k$. 

Notice that $g_{r}(v_j)g_{r}(v_j^F)\leq 0$ holds for any $j=2,\ldots,m$. Therefore, the function $g_r$ has at least $m$ strong nodal domains on $T$. This implies  $g_r$ has at least $m-\ell_G$ strong nodal domains on $G$. Since $g_r$ is an eigenfunction of $A_G$ corresponding to $\lambda_k(G)$, we have $m-\ell_G\leq (k-1)\Delta_G$ by Theorem \ref{thm:nodal upper}. This concludes the proof.
  \end{proof}


Let $G=(V,E)$ be a graph. Given two vertex sets $V_1,V_2\subset V$, we say $V_1,V_2$ are edge-disjoint if $\{\{i,j\}\in E:i\in V_1,j\in V_2\}=\emptyset $. To complete the proof of Theorem \ref{thm:multi}, we need the following lemma.
\begin{lemma}[{\cite[Lemma 2.2]{balla2024equiangular-exponential-regime}}]\label{lemma:edge disjoint}
    Let $G=(V,E)$ be a connected graph and $V_1,V_2\subset V$. If $V_1$ and $V_2$ are edge-disjoint, we have $\lambda_1(G[V_1])<\lambda_2(G)$ or $\lambda_1(G[V_2])<\lambda_2(G)$ or $\lambda_1(G[V_1])=\lambda_1(G[V_2])=\lambda_2(G).$
\end{lemma}

\begin{proof}[Proof of Theorem \ref{thm:multi}]
Since $m_{\lambda_2(G)}\leq |V|$, it remains to consider $|V|\geq 2n_{\lambda}\Delta^{n_{\lambda}+3}\left(1+\Delta+\Delta^2\right)$. If $G$ is a tree, we have $m_{\lambda_{2}(G)}\leq \Delta_{G}<2n_{\lambda}\Delta^{n_{\lambda}+3}\left(1+\Delta+\Delta^2\right)$ according to Lemma \ref{lemma:multi tree}. Next, we assume that $G$ is not a tree.

Case $1$: $\mathfrak{g}(G)\leq2n_{\lambda}$. We take a cycle $\mathcal{C}$ in $G$ such that the length of $\mathcal{C}$ is $\mathfrak{g}(G)$. Since $|V|\geq 2n_{\lambda}\Delta^{n_\lambda}$, there must exist two vertices $x\in \mathcal{C}$ and $y\in V$ such that $\mathrm{dis}(y,\mathcal{C})=\mathrm{dis}(x,y)=n_{\lambda}$. This implies that $\mathcal{H}(\mathfrak{g}(G),n_{\lambda})$ is a subgraph of $G[B^G_{n_{\lambda}}(\mathcal{C})]$. We have $$\lambda_1\left( G[B^G_{n_{\lambda}}(\mathcal{C})]\right)\geq \lambda_1\left( \mathcal{H}(\mathfrak{g}(G),n_{\lambda})\right)\geq \lambda_1(\mathcal{H}(2n_\lambda+1),n_{\lambda})\geq \lambda_1(\mathcal{T}(n_\lambda,n_\lambda,n_\lambda))> \lambda,$$
where the first inequality follows from Lemma \ref{lemma:mono1}, the second inequality follows from Lemma \ref{lemma:mono2} and the third inequality is derived from the fact that $\mathcal{T}(n_\lambda,n_\lambda,n_\lambda)$ is a subgraph of $\mathcal{H}(2n_{\lambda}+1,n_\lambda)$ and Lemma \ref{lemma:mono1}. By Lemma \ref{lemma:edge disjoint}, we have \[\lambda_1\left(G[V\setminus B^G_{n_{\lambda}+2}(\mathcal{C})]\right)<\lambda.\]
Denote $\{G_i\}_{i=1}^s$ as the connected components of $G[V\setminus B^G_{n_{\lambda}+2}(\mathcal{C})]$. Fix an $i\in\{1,\ldots,s\}$. If $\ell_{G_i}\geq 2$, then $G_i$ must contain a graph of the class defined in Definition \ref{def:4graph} $(a)$, $(b)$ or $(c)$. By Proposition \ref{pro:two-cycle} and Lemma \ref{lemma:mono1}, we have $$\lambda_1\left(G[V\setminus B^G_{n_{\lambda}+2}(\mathcal{C})]\right)\geq \lambda_1\left(G_i\right)\geq \frac{3}{\sqrt{2}}>\lambda.$$ Contradiction. This implies $\ell_{G_i}\leq 1$ for any $1\leq i \leq s$. Since $\Delta_G\leq \Delta$, we have $$s\leq \Delta \left|B^G_{n_{\lambda}+2}(\mathcal{C})\right|\leq 2n_{\lambda}\Delta^{n_{\lambda}+4}.$$ By Lemma \ref{lemma:multi tree} and Theorem \ref{thm:interlacing}, we have $$m_{\lambda_2(G)}\leq 2(\Delta+1)n_{\lambda}\Delta^{n_{\lambda}+4}+2n_{\lambda}\Delta^{n_{\lambda}+3}=2n_{\lambda}\Delta^{n_{\lambda}+3}\left(1+\Delta+\Delta^2\right).$$   

Case $2$: $\mathfrak{g}(G)>2n_{\lambda}$. If $\ell_G\leq 1$, then we have $m_{\lambda_2(G)}\leq \Delta+1<2n_{\lambda}\Delta^{n_{\lambda}+3}\left(1+\Delta+\Delta^2\right)$. 

If $\ell_G\geq 2$, then $G$ must contain a graph of the class defined in Definition \ref{def:4graph} $(a)$, $(b)$ or $(c)$. Since $\mathfrak{g}(G)>2n_{\lambda}$, we have that $\mathcal{T}(n_\lambda,n_\lambda,n_\lambda)$ is a subgraph of $G$. Let $x\in V$ be a vertex such that $\mathcal{T}(n_\lambda,n_\lambda,n_\lambda)$ is a subgraph of $G[B^G_{n_\lambda}(x)]$. This implies $\lambda_1\left(G[ B^G_{n_\lambda}(x)]\right)>\lambda$. By Lemma \ref{lemma:edge disjoint}, we have $\lambda_1\left(G[V\setminus B^G_{n_\lambda+2}(x)]\right)<\lambda$. Denote the connected components of $G[V\setminus B^G_{n_{\lambda}+2}(x)]$ by $\{G_i\}_{i=1}^t$. Similarly to Case $1$, we have $t\leq \Delta^{n_{\lambda}+4}$ and $\ell_{G_i}<2$ for any $1\leq i\leq t$. By Lemma \ref{lemma:multi tree} and Theorem \ref{thm:interlacing}, we have $$m_{\lambda_2(G)}\leq \Delta^{n_{\lambda+3}}\left(1+\Delta+\Delta^2\right).$$

Combining the conclusions from Case $1$ and Case $2$, we derive that $$m_{\lambda_2(G)}\leq 2n_{\lambda}\Delta^{n_{\lambda}+3}\left(1+\Delta+\Delta^2\right).$$This concludes the proof.
\end{proof}

\section{Proof of Theorem \ref{thm:main}}\label{sec:equiangular}
Following the proof of Theorem \ref{thm:k-lambad} due to \cite{Jiang-Tidor-Yao-Zhang-Zhao-21}, we can directly derive Theorem \ref{thm:main} by applying Theorem \ref{thm:multi}. For the convenience of the reader, we present the details here.

We first prepare the following two results. The first is a combination of \cite[Lemma 3.1]{Jiang-Tidor-Yao-Zhang-Zhao-21} and \cite[Lemma 10]{balla2024equiangular-exponential-regime}.
\begin{lemma}\label{lemma:equi-and-graph}
    There exist $N$ equiangular lines in $\mathbb{R}^d$ with common angle $\arccos\,\alpha$ if and only if there exists an $N$-vertex graph $G$ with  $\Delta_G\leq \frac{6}{\alpha^4}$ such that the matrix $\lambda I-A_G+\frac{1}{2}J$ is positive semidefinite and has rank at most $d$, where $\lambda=\frac{1-\alpha}{2\alpha}$ and $J$ is the $N\times N$ all-$1$ matrix.
\end{lemma}
\begin{pro}[{\cite[Proposition 3.2]{Jiang-Tidor-Yao-Zhang-Zhao-21}}]\label{pro:lower bound}
    Let $\alpha\in (0,1)$ and $d\in Z^+$. Define $\lambda=\frac{1-\alpha}{2\alpha}$. We have $N_{\alpha}(d)\geq d$. Moreover, if $k(\lambda)<\infty$, we have  $N_{\alpha}(d)\geq \left\lfloor \frac{(d-1)\kappa}{\kappa-1}\right\rfloor$.
\end{pro}
\begin{proof}[Proof of Theorem \ref{thm:main}]
    (a) By Proposition \ref{pro:lower bound}, it is sufficient to prove $N_{\alpha}(d)\leq \left\lfloor \frac{(d-1)\kappa}{\kappa-1}\right\rfloor$.
    
    We recall that $n_{\lambda_0}$ is the smallest integer such that $\mathcal{T}\left(n_{\lambda_0},n_{\lambda_0},n_{\lambda_0}\right)>\lambda_0$. 
     Assume that there are $N$ equiangular lines in $\mathbb{R}^d$ with common
angle $\arccos \alpha$. By Lemma \ref{lemma:equi-and-graph}, there exists an  $N$-vertex graph $G$ with $\Delta_G\leq 6(2\lambda_0+1)^4:=\Delta_{\lambda_0}$ such that the matrix $\lambda I-A_{G}+\frac{1}{2}J$ is positive semidefinite and has rank at most $d$.
Let us assume 
\begin{equation}\label{eq:d_assumption}
d\geq \left[2+2n_{\lambda_0}\Delta_{\lambda_0}^{n_{\lambda_0}+3}\left(1+\Delta_{\lambda_0}+\Delta_{\lambda_0}^2\right)\right](\kappa-1)+\kappa.
\end{equation}

    We denote the connected components of $G$ by $G_1,G_2,\ldots,G_t$ and assume $\lambda_1(G_1)=\lambda_1(G)$.

    Case $1$: $\lambda$ is not an eigenvalue of $A_G$.
    
    We have $$d\geq \mathrm{rank}(\lambda I-A_G+\frac{1}{2}J)\geq N-1,$$
    since the matrix $\lambda I-A_G$ has full rank and $J$ has rank $1$. This implies $$N\leq d+1<  \left\lfloor \frac{(d-1)\kappa}{\kappa-1}\right\rfloor.$$ Contradiction.

    Case $2$: $\lambda_1(G)=\lambda$. 
    
    Assume $\lambda_1(G_1)=\lambda_1(G_2)=\cdots\lambda_1(G_j)=\lambda$.  By Perron-Frobenius theorem, we obtain $$\mathrm{dim\,ker}(\lambda I-A_G)=j.$$ Since $\lambda I-A_G$ and $J$ are positive semidefinite, we have $$\mathrm{ker}(\lambda I-A_G+\frac{1}{2}J)=\ker(\lambda I-A_G)\cap\ker(J)$$ 
 By Perron-Frobenius theorem, we have that there must exist a nonnegative vector $f\in \ker(\lambda I-A_G)$. We have $f\notin \ker(J)$ by direct computation. This implies 
 $$\dim\left(\ker(\lambda I-A_G)\cap\ker(J)\right)\leq \dim\ker(\lambda I-A_G)-1.$$ 
By rank-nullity theorem, we have  \begin{equation}\label{eq:d-1}
    \mathrm{rank}(\lambda I-A_G)\leq \mathrm{rank}(\lambda I-A_G+J)-1\leq d-1.\end{equation}
Thus, $$N=\ker(\lambda I-A_G)+\mathrm{rank}(\lambda I-A_G)\leq d-1+j.$$

By definition of $\kappa$, we have $|V(G_i)|\geq \kappa$, where $|V(G_i)|$ is the number of vertex of $G_i$. This implies $\mathrm{rank}(\lambda I-A_G)\geq \kappa j-j.$ Combining this inequality with (\ref{eq:d-1}), we have $j\leq \frac{d-1}{\kappa-1}$. Then we have $$N\leq d-1-j\leq \frac{(d-1)\kappa}{\kappa-1}.$$

Case $3$: $\lambda_1(G)>\lambda$ with $\lambda$ being an eigenvalue of $A_G$. We have $\lambda_2(G)\leq \lambda$ because $\lambda I-A_G+\frac{1}{2}J$ is positive semidefinite and $J$ has rank $1$. We first prove $$\mathrm{dim\,ker}(\lambda I-A_G)=\mathrm{dim\,ker}(\lambda I-A_{G_1}).$$

Assume there exists $j>1$ such that $\lambda_1(G_j)=\lambda$. Let $f_1$ be an eigenfunction of $A_G$ corresponding to $\lambda_1(G)$ and $f_2$ be an eigenfunction of $A_G$ corresponding to $\lambda$. By the Perron-Frobenius theorem, we can assume that $f_1$ is positive on the vertices of $G_1$ and $0$ on the other vertices, and $f_2$ is positive on the vertices of $G_j$ and $0$ on the other vertices. Moreover, we assume $(f_1)^Tf_1=(f_2)^Tf_2=1$.

Take a constant $\gamma$ such that $1^T(f_1-\gamma f_2)=0$. Because $\lambda I-A_G+\frac{1}{2}J$ is positive semidefinite, we compute $$0\leq (f_1-\gamma f_2)^T(\lambda I-A_G+\frac{1}{2}J)(f_1-\gamma f_2)\leq \lambda-\lambda_1(G_1)<0,$$
contradiction.
This implies $\mathrm{dim\,ker}(\lambda I-A_G)=\mathrm{dim\,ker}(\lambda I-A_{G_1}).$ By Theorem \ref{thm:multi}, we have $$\mathrm{dim\,ker}(\lambda I-A_{G_1})\leq 2n_{\lambda}\Delta^{n_{\lambda}+3}\left(1+\Delta+\Delta^2\right).$$
 Since $\mathrm{rank}(\lambda I-A_G)\leq \mathrm{rank}(\lambda I-A_G+\frac{1}{2}J)+1\leq d+1$, we compute \begin{equation*}
\begin{aligned}
    N&=\mathrm{rank}(\lambda I-A_G)+\mathrm{dim \,ker}(\lambda I-A_G)=\mathrm{rank}(\lambda I-A_G)+\mathrm{dim \,ker}(\lambda I-A_{G_1})\\
    &\leq d+1+2n_{\lambda}\Delta^{n_{\lambda}+3}\left(1+\Delta+\Delta^2\right)\\
    &<\left\lfloor\frac{(d-1)\kappa}{\kappa-1}\right\rfloor,
    \end{aligned}
\end{equation*}
where the last inequality follows from our assumption \eqref{eq:d_assumption} on $d$. Contradiction.

Combining the results in Case $1$, Case $2$ and Case $3$, we have  $N\leq\left\lfloor\frac{(d-1)\kappa}{\kappa-1}\right\rfloor$ holds for all $d$ satisfying \eqref{eq:d_assumption}. This completes the proof of Theorem \ref{thm:main} (a).

(b) Since $\kappa=+\infty$, we have $\lambda_1(G)\neq \lambda$, that is, the above Case $2$ is not feasible.
Similarly to Case $1$ and Case $3$ above, we have that 
\begin{equation*}
    N_{\alpha}(d)\leq d+1+2n_{\lambda}\Delta^{n_{\lambda}+3}\left(1+\Delta+\Delta^2\right).
\end{equation*}
By Proposition \ref{pro:lower bound}, we have $N_{\alpha}(d)=d+O_{\lambda_0}(1)$.
\end{proof}

\section*{Acknowledgement}

This work is supported by the National Key R and D Program of China 2023YFA1010200, the National Natural Science Foundation of China (No. 12431004). We are very grateful to Yanlong Ding for helpful discussions.

\bibliographystyle{plain}
\bibliography{references}
\end{document}